\pdfoutput=1
\documentclass[11pt]{article}

\usepackage[T1]{fontenc}
\usepackage[utf8]{inputenc}
\usepackage{amsmath,amssymb,amsthm,mathtools}
\usepackage{libertinus-type1}
\usepackage{libertinust1math}
\usepackage[margin=1in]{geometry}
\usepackage{xcolor}
\usepackage{enumitem}
\usepackage{booktabs}
\usepackage{microtype}
\usepackage[colorlinks=true,linkcolor=blue!50!black,citecolor=blue!50!black,urlcolor=blue!50!black]{hyperref}
\hypersetup{
  pdftitle={Chow classes of orbit closures of skew-symmetric matrix pencils},
  pdfauthor={Ari Krishna},
  pdfsubject={Algebraic geometry},
  pdfkeywords={skew-symmetric matrix pencil, Pfaffian cubic, orbit closure, Chow class, Schubert calculus, Jordan-Kronecker form}
}

\usepackage{aliascnt}
\newtheorem{theorem}{Theorem}[section]
\newaliascnt{proposition}{theorem}
\newtheorem{proposition}[proposition]{Proposition}
\aliascntresetthe{proposition}
\newaliascnt{corollary}{theorem}

\aliascntresetthe{corollary}
\newaliascnt{lemma}{theorem}
\newtheorem{lemma}[lemma]{Lemma}
\aliascntresetthe{lemma}
\theoremstyle{definition}
\newaliascnt{definition}{theorem}

\aliascntresetthe{definition}
\newaliascnt{remark}{theorem}

\aliascntresetthe{remark}
\newaliascnt{example}{theorem}

\aliascntresetthe{example}
\usepackage[nameinlink,noabbrev]{cleveref}
\crefname{theorem}{theorem}{theorems}
\Crefname{theorem}{Theorem}{Theorems}
\crefname{proposition}{proposition}{propositions}
\Crefname{proposition}{Proposition}{Propositions}
\crefname{corollary}{corollary}{corollaries}
\Crefname{corollary}{Corollary}{Corollaries}
\crefname{lemma}{lemma}{lemmas}
\Crefname{lemma}{Lemma}{Lemmas}
\crefname{definition}{definition}{definitions}
\Crefname{definition}{Definition}{Definitions}
\crefname{remark}{remark}{remarks}
\Crefname{remark}{Remark}{Remarks}
\crefname{example}{example}{examples}
\Crefname{example}{Example}{Examples}

\newcommand{\Z}{\mathbb Z}
\newcommand{\Pj}{\mathbb P}
\newcommand{\Gr}{\operatorname{Gr}}
\newcommand{\CH}{\operatorname{CH}}
\newcommand{\codim}{\operatorname{codim}}
\newcommand{\rk}{\operatorname{rk}}
\newcommand{\Pf}{\operatorname{Pf}}
\newcommand{\Sym}{\operatorname{Sym}}
\newcommand{\Span}{\operatorname{Span}}
\newcommand{\cO}{\mathcal O}
\newcommand{\cS}{\mathcal S}
\newcommand{\cQ}{\mathcal Q}
\newcommand{\cE}{\mathcal E}
\newcommand{\cK}{\mathcal K}
\newcommand{\cA}{\mathcal A}
\newcommand{\cR}{\mathcal R}
\newcommand{\cD}{\mathcal D}
\newcommand{\cF}{\mathcal F}

\newcommand{\cT}{\mathcal T}

\newcommand{\typeaaa}{\ensuremath{\mathbf{111}}}
\newcommand{\typetwoone}{\ensuremath{\mathbf{21}}}
\newcommand{\typethree}{\ensuremath{\mathbf{3}}}
\newcommand{\typeonetwo}{\ensuremath{\mathbf{(11)1}}}
\newcommand{\typetwopart}{\ensuremath{\mathbf{(21)}}}

\title{Chow classes of orbit closures of skew-symmetric matrix pencils}
\author{Ari Krishna}

\begin{document}
\maketitle

\begin{abstract}
Let $V$ be a six-dimensional complex vector space and let
\[
G=\Gr(2,\Lambda^2V^\vee)
\]
parametrize pencils of skew-symmetric $6\times6$ matrices.  The group $\operatorname{PGL}(V)$ has $11$ orbits on $G$, classified by the Jordan--Kronecker canonical form.  We compute the Chow class of every orbit closure in the Schubert basis of $\CH^*(G)$. 
\end{abstract}



\section{Introduction}

Let $V$ be a complex vector space of dimension six and put
\[
W=\Lambda^2V^\vee,\qquad G=\Gr(2,W).
\]
A point $L\in G$ is an unframed pencil of alternating forms on $V$.  Choosing a basis $A,B$ of $L$ produces a skew-symmetric matrix pencil
\[
A-\lambda B,
\]
well-defined up to simultaneous congruence by $\operatorname{GL}(V)$ and change of basis in $L$.  The congruence classification of skew-symmetric matrix pencils is classical and is expressed in terms of Jordan--Kronecker blocks.  The dimensions and closure relations of the corresponding affine congruence orbits have been studied systematically by Dmytryshyn, K\aa gstr\"om, Sergeichuk, and Das \cite{DKS13,DK14,DD26}.  Passing to unframed pencils identifies eigenvalue configurations up to $\operatorname{PGL}_2$.  In size six, there are at most three eigenvalues, and consequently the induced action
\[
\operatorname{PGL}(V)\curvearrowright G
\]
has finitely many orbits.

The projective geometry is governed by the Pfaffian cubic hypersurface
\[
\Pf(V)=\{[\omega]\in\Pj(W):\rk(\omega)\leq4\}.
\]
Its singular locus is the Pl\"ucker Grassmannian
\[
X=\Gr(2,V^\vee)=\{[\omega]:\rk(\omega)\leq2\}\subset\Pj(W).
\]
For a pencil $L$, restriction of the Pfaffian gives a binary cubic
\[
\operatorname{pf}_L\in\Sym^3L^\vee\otimes\det(V^\vee).
\]
The root multiplicities of this cubic describe the regular Jordan data coarsely, while incidence with $X$ and the common kernel
\[
\ker L:=\bigcap_{\omega\in L}\ker(\omega)
\]
recover the finer Jordan--Kronecker type.

Our purpose is to compute all eleven fundamental classes in
\[
\CH^*(G)=\bigoplus_{13\geq a\geq b\geq0}\Z\sigma_{a,b},
\qquad
\sigma_{a,b}=s_{a,b}(\cS^\vee),
\]
where $\cS$ is the tautological rank-two bundle on $G$.  The main result is the following table.  The orbit labels and canonical forms are recalled in \cref{sec:orbits}.

\begin{theorem}\label{thm:main}
The eleven $\operatorname{PGL}(V)$-orbit closures in $G=\Gr(2,\Lambda^2V^\vee)$ have the following Chow classes.
\[
\begin{array}{ccl}
\hline
\text{type} & \codim_G & [\overline{\cO}]\\
\hline
\typeaaa &0&1\\
\typetwoone &1&6\sigma_1\\
\typethree &2&6\sigma_2+9\sigma_{1,1}\\
M_1\oplus M_1 &4&18\sigma_{3,1}+27\sigma_{2,2}\\
\typeonetwo &5&14\sigma_5\\
M_2\oplus M_0 &5&12\sigma_{4,1}+30\sigma_{3,2}\\
\typetwopart &6&14\sigma_6+28\sigma_{5,1}\\
M_1\oplus M_0\oplus H_1 &7&42\sigma_{6,1}+84\sigma_{5,2}\\
M_0^2\oplus H_1\oplus H_1 &10&3\sigma_{9,1}+15\sigma_{8,2}+40\sigma_{7,3}+70\sigma_{6,4}+91\sigma_{5,5}\\
M_0^2\oplus H_2 &11&6\sigma_{10,1}+24\sigma_{9,2}+50\sigma_{8,3}+60\sigma_{7,4}+42\sigma_{6,5}\\
M_1\oplus M_0^3 &15&56\sigma_{9,6}+70\sigma_{8,7}.\\
\hline
\end{array}
\]
\end{theorem}

Several entries admit short intrinsic proofs.  The repeated-root divisor and the triple-root locus are obtained from relative jets of the universal Pfaffian binary cubic.  The identically singular locus is the zero scheme of the universal section of $\Sym^3\cS^\vee$; Beauville proved that this Fano scheme of lines on the Pfaffian cubic has the expected codimension and is irreducible \cite{Bea05}.  The class $14\sigma_5$ is the incidence class of lines meeting the rank-two Grassmannian $X$, whose Pl\"ucker degree is fourteen.

The remaining classes are obtained from smooth incidence spaces.  Common-kernel strata are dominated birationally by relative Grassmannians
\[
\Gr_{\Gr(k,V)}\bigl(2,\Lambda^2(V/\cK)^\vee\bigr),
\qquad k=1,2,3,
\]
where $\cK$ is the tautological bundle. The two refinements are imposed by a marked rank-two member or by the discriminant of the quotient Pfaffian.  Their Schubert coefficients are computed by iterated Gysin maps. The resulting numbers are Chern numbers of exterior-square bundles on $\Gr(k,6)$, and we evaluate them by a uniform coefficient-extraction formula.  

\section{Jordan--Kronecker types}\label{sec:orbits}

We recall the canonical blocks in the form needed here. Let $J_h(\mu)$ be the $h\times h$ Jordan block with eigenvalue $\mu$, and put
\[
H_h(\mu)=
\begin{bmatrix}0&J_h(\mu)\\-J_h(\mu)^t&0\end{bmatrix}
-\lambda
\begin{bmatrix}0&I_h\\-I_h&0\end{bmatrix}.
\]
The analogous block with eigenvalue at infinity will be denoted $K_h$.  Let
\[
F_m=[I_m\;0],\qquad G_m=[0\;I_m]
\]
and define the singular block of size $2m+1$ by
\[
M_m=
\begin{bmatrix}0&F_m\\-F_m^t&0\end{bmatrix}
-\lambda
\begin{bmatrix}0&G_m\\-G_m^t&0\end{bmatrix}.
\]
In particular, $M_0$ is the one-dimensional zero block.  Every skew-symmetric pencil is congruent to a direct sum of these blocks, uniquely up to the usual permutation and eigenvalue data \cite{DKS13}.

A change of basis in the two-dimensional pencil acts on the eigenvalues by $\operatorname{PGL}_2$.  Since a six-dimensional regular alternating pencil has Pfaffian degree three, any three distinct eigenvalues may be normalized to $0,1,\infty$.  The following finite list is immediate from the canonical form.

\begin{proposition}\label{prop:orbits}
The action of $\operatorname{PGL}(V)$ on $G$ has eleven orbits.  Representatives and geometric descriptions are as follows.
\[
\begin{array}{cll}
\hline
\text{label}&\text{canonical type}&\text{generic geometric description}\\
\hline
\typeaaa&H_1(0)\oplus H_1(1)\oplus K_1&\operatorname{pf}_L\text{ squarefree}\\
\typetwoone&H_2(0)\oplus K_1&\text{one double root at a rank-four form}\\
\typeonetwo&H_1(0)^2\oplus K_1&\text{one rank-two member and one simple root}\\
\typethree&H_3(0)&\text{a triple root at a rank-four form}\\
\typetwopart&H_2(0)\oplus H_1(0)&\text{a triple root at a rank-two form}\\
M_1\oplus M_1&M_1\oplus M_1&\operatorname{pf}_L\equiv0,\ \ker L=0\\
M_2\oplus M_0&M_2\oplus M_0&\dim\ker L=1,\text{ constant quotient rank }4\\
M_1\oplus M_0\oplus H_1&M_1\oplus M_0\oplus H_1(0)&\dim\ker L=1,\text{ quotient has a rank-two member}\\
M_0^2\oplus H_1\oplus H_1&M_0^2\oplus H_1(0)\oplus K_1&\dim\ker L=2,\text{ quotient Pfaffian squarefree}\\
M_0^2\oplus H_2&M_0^2\oplus H_2(0)&\dim\ker L=2,\text{ quotient Pfaffian has a double root}\\
M_1\oplus M_0^3&M_1\oplus M_0^3&\dim\ker L=3.\\
\hline
\end{array}
\]
Their codimensions are those displayed in \cref{thm:main}.
\end{proposition}

\begin{proof}
For regular pencils the sum of the sizes of the $H$- and $K$-blocks is six, hence the corresponding block sizes form a partition of three.  The decompositions with linearly independent coefficient matrices are exactly the five regular types listed above.  The nominal type $H_1(0)^3$ has proportional coefficient matrices and therefore does not define a point of $G$.

For singular pencils, the number of odd-dimensional $M$-blocks is even.  Enumerating direct sums of total size six gives
\[
M_1^2,\quad M_2M_0,\quad M_1M_0H_1,\quad
M_0^2H_2,\quad M_0^2H_1H_1,\quad M_1M_0^3,
\]
after excluding the types for which the two coefficient matrices are proportional.  The geometric descriptions follow directly from the kernels and elementary divisors of the blocks.  The codimensions may be read from the incidence models constructed below; they also agree with the congruence-orbit formulas of \cite{DKS13} after quotienting by changes of basis in the pencil.
\end{proof}

\section{Schubert conventions and the Pfaffian cubic}

Let
\[
0\longrightarrow\cS\longrightarrow W\otimes\cO_G\longrightarrow\cQ\longrightarrow0
\]
be the tautological sequence on $G=\Gr(2,W)$.  Put
\[
a=c_1(\cS^\vee),\qquad b=c_2(\cS^\vee).
\]
If $x,y$ are the Chern roots of $\cS^\vee$, our Schubert convention is
\[
\sigma_{p,q}=s_{p,q}(x,y)=(xy)^q\sum_{j=0}^{p-q}x^{p-q-j}y^j,
\qquad13\geq p\geq q\geq0.
\]
The Poincar\'e dual of $\sigma_{p,q}$ is
\[
\sigma_{13-q,13-p}.
\]

The Pfaffian is a cubic polynomial on $W=\Lambda^2V^\vee$ with values in the fixed line $\det V^\vee$.  Restriction to the universal pencil gives a section
\[
\operatorname{pf}\in H^0(G,\Sym^3\cS^\vee)\otimes\det V^\vee.
\]
We suppress the fixed determinant line in ordinary Chow calculations.

\subsection{Gysin formulas}\label{sec:gysin}

We collect the pushforward identities that will be used repeatedly.  Our projective bundles parametrize one-dimensional subspaces.  Thus if
\[
p:\Pj(\cF)\longrightarrow B
\]
is the projective bundle of lines in a rank-$N$ bundle $\cF$, if $\cD\subset p^*\cF$ is the tautological line, and if
\[
\xi=c_1(\cD^\vee),
\]
then
\begin{equation}\label{eq:proj-gysin}
p_*(\xi^{N-1+j})=s_j(\cF^\vee),\qquad j\geq0.
\end{equation}
Here and below $s_j$ denotes the one-row Schur class, equivalently the complete symmetric polynomial in the Chern roots.  Formula \eqref{eq:proj-gysin} follows from the projective-bundle relation and is compatible with the convention $p_*(\xi^{N-1})=1$.

We also use the following Grassmann-bundle form of the Schur Gysin formula.  Let $\cF$ have rank $N$, let
\[
p:\Gr_B(2,\cF)\longrightarrow B
\]
parametrize rank-two subbundles, and let $\cT$ be its tautological bundle.

\begin{lemma}[Schur Gysin formula]\label{lem:schur-gysin}
For $u\geq v\geq0$,
\[
p_*s_{u,v}(\cT^\vee)=
\begin{cases}
s_{u-N+2,\,v-N+2}(\cF^\vee),&v\geq N-2,\\
0,&v<N-2,
\end{cases}
\]
where a zero part is omitted.
\end{lemma}

\begin{proof}
The relative dimension is $2(N-2)$.  The classes
\[
s_{\lambda+(N-2,N-2)}(\cT^\vee)
\]
are the relative Schubert classes dual to $s_\lambda(\cF^\vee)$ under the Grassmann-bundle pairing.  Equivalently, the formula is the rank-two case of the standard Gysin identity for Schur classes on a Grassmann bundle; see \cite[Theorem~1.1]{DP18}.  The vanishing for $v<N-2$ follows because the partition does not contain the fiberwise top rectangle $(N-2,N-2)$.
\end{proof}

For the final integrations on ordinary Grassmannians, we use an explicit coefficient formula.  Let $\cE$ be the universal quotient bundle of rank $r=6-k$ on $\Gr(k,6)$, and let $x_1,\ldots,x_r$ denote its formal Chern roots.  If $\Phi$ is symmetric and homogeneous of degree $kr$, then
\begin{equation}\label{eq:grass-coeff}
\int_{\Gr(k,6)}\Phi(\cE)
=
[x_1^5x_2^4\cdots x_r^{6-r}]
\left(
\Phi(x_1,\ldots,x_r)
\prod_{1\leq i<j\leq r}(x_i-x_j)
\right).
\end{equation}
Indeed, after expanding $\Phi$ in Schur polynomials, the alternant formula shows that the indicated coefficient is the coefficient of the top rectangular Schur class $s_{(k^r)}(\cE)$, whose integral is one.

In particular, for $\Lambda^2\cE$ we evaluate Schur classes in the alphabet
\[
\{x_i+x_j:1\leq i<j\leq r\}.
\]
The required finite calculations may be performed from
\begin{equation}\label{eq:complete-generating}
\sum_{m\geq0}h_m\bigl(\Lambda^2\cE\bigr)t^m
=
\prod_{i<j}\frac{1}{1-(x_i+x_j)t}
\end{equation}
and the two-row Jacobi--Trudi identity
\begin{equation}\label{eq:JT-two-row}
s_{a,b}=h_ah_b-h_{a+1}h_{b-1}.
\end{equation}
Equations \eqref{eq:grass-coeff}--\eqref{eq:JT-two-row} give a direct verification of every numerical table below.

\subsection{Repeated and triple roots}

Let
\[
\pi:\Pj(\cS)\longrightarrow G,
\qquad h=c_1\bigl(\cO_{\Pj(\cS)}(1)\bigr).
\]
The universal Pfaffian is a section of $\cO_{\Pj(\cS)}(3)$.  Since
\[
c_1(\Omega_\pi)=-2h+a,
\]
the marked multiplicity-$m$ locus has class
\[
\prod_{i=0}^{m-1}\bigl(3h+i(-2h+a)\bigr).
\]
We use
\[
\pi_*(h)=1,\qquad \pi_*(h^2)=a,\qquad \pi_*(h^3)=a^2-b.
\]

\begin{proposition}\label{prop:jets}
The closures of the types \typetwoone and \typethree have classes
\[
[\overline{\cO_{\typetwoone}}]=6\sigma_1,
\qquad
[\overline{\cO_{\typethree}}]=6\sigma_2+9\sigma_{1,1}.
\]
\end{proposition}

\begin{proof}
For a marked double root,
\[
\pi_*\bigl(3h(h+a)\bigr)=3a+3a=6a.
\]
For a marked triple root,
\[
\begin{aligned}
\pi_*\bigl(3h(h+a)(-h+2a)\bigr)
&=\pi_*(-3h^3+3ah^2+6a^2h)\\
&=6a^2+3b\\
&=6\sigma_2+9\sigma_{1,1}.
\end{aligned}
\]
A general form in each locus has a unique marked multiple root, so the marked incidence map is generically one-to-one.
\end{proof}

\subsection{Identically singular pencils}

Let
\[
\cF=\{L\in G:\operatorname{pf}_L\equiv0\}.
\]
This is the Fano scheme of lines contained in the Pfaffian cubic.  Beauville proved that for alternating forms on a six-dimensional space it is irreducible of the expected codimension four \cite{Bea05}.  Hence it is the zero scheme of a regular section of $\Sym^3\cS^\vee$.

\begin{proposition}\label{prop:pfzero}
The closure of the type $M_1\oplus M_1$ is $\cF$, and
\[
[\cF]=18\sigma_{3,1}+27\sigma_{2,2}.
\]
\end{proposition}

\begin{proof}
The Chern roots of $\Sym^3\cS^\vee$ are
\[
3x,\quad 2x+y,\quad x+2y,\quad 3y.
\]
Thus,
\[
\begin{aligned}
c_4(\Sym^3\cS^\vee)
&=9xy(2x+y)(x+2y)\\
&=18a^2b+9b^2\\
&=18\sigma_{3,1}+27\sigma_{2,2}.
\end{aligned}
\]
The common-kernel locus has codimension five by \cref{prop:common-kernel}; hence the general point of the irreducible four-codimensional locus $\cF$ has zero common kernel and therefore type $M_1\oplus M_1$.
\end{proof}

\section{Pencils containing a rank-two form}

The rank-two locus in $\Pj(W)$ is the Pl\"ucker Grassmannian
\[
X=\Gr(2,V^\vee)\subset\Pj(\Lambda^2V^\vee)=\Pj^{14}.
\]
It has codimension six and degree fourteen.  Let $\cA$ denote its tautological rank-two bundle, and let
\[
\cR=\Lambda^2\cA\subset W\otimes\cO_X
\]
be the tautological line of decomposable forms.  Put
\[
Z=\Pj\bigl((W\otimes\cO_X)/\cR\bigr).
\]
A point of $Z$ is a rank-two form $[q]\in X$ together with a direction $[r]$ modulo $q$, and therefore determines the line $\Pj\Span(q,r)$.  We obtain a proper map
\[
f:Z\longrightarrow G.
\]
It is birational over the locus of lines meeting $X$ in one reduced point.

\begin{proposition}\label{prop:ranktwo}
The closure of the type \typeonetwo is the locus of lines meeting $X$, and
\[
[\overline{\cO_{\typeonetwo}}]=14\sigma_5.
\]
\end{proposition}

\begin{proof}
Let $\rho:\Pj(\cS)\to G$ be the universal line and let $\mathrm{ev}:\Pj(\cS)\to\Pj(W)$ be evaluation.  Since
\[
[X]=14H^6\in\CH^6(\Pj(W)),
\]
the incidence class is
\[
\rho_*\mathrm{ev}^*[X]=14\rho_*(h^6)=14s_5(\cS^\vee)=14\sigma_5.
\]
A general line meeting $X$ is not contained in the Pfaffian cubic, and has canonical type \typeonetwo.
\end{proof}

At a decomposable form $q$, let $K=\ker q$, of dimension four.  The Pfaffian expansion along the line $q+tr$ begins
\[
\Pf(q+tr)=t^2\,q\wedge r\wedge r+t^3\,r\wedge r\wedge r,
\]
up to a fixed nonzero scalar. Hence, the root at $q$ is triple precisely when the restriction $r|_K$ is degenerate.  On $Z$, if $\cD\subset ((W\otimes\cO_X)/\cR)$ is the tautological direction line, this condition is the zero divisor of
\[
(\cD^\vee)^2\otimes\det(K^\vee).
\]

\begin{proposition}\label{prop:ranktwo-triple}
The closure of the type \typetwopart has class
\[
[\overline{\cO_{\typetwopart}}]=14\sigma_6+28\sigma_{5,1}.
\]
\end{proposition}

\begin{proof}
Let
\[
r=c_1(\cR^\vee),\qquad \xi=c_1(\cD^\vee).
\]
If $K\subset V\otimes\cO_X$ is the rank-four common kernel of the tautological rank-two form, then the exact sequence
\[
0\longrightarrow K\longrightarrow V\otimes\cO_X\longrightarrow\cA^\vee\longrightarrow0
\]
shows, up to the constant line $\det V^\vee$, that
\[
c_1(\det K^\vee)=c_1(\det\cA^\vee)=r.
\]
Consequently, the triple-root divisor on $Z$ has class $2\xi+r$.

Put
\[
\cF=(W\otimes\cO_X)/\cR,
\]
which has rank fourteen.  The pullback of the ambient tautological bundle has a filtration with line quotients $\cR$ and $\cD$, so
\[
f^*\sigma_{p,q}=s_{p,q}(r,\xi).
\]
Moreover, dualizing $0\to\cR\to W\otimes\cO_X\to\cF\to0$ gives
\[
0\longrightarrow\cF^\vee\longrightarrow W^\vee\otimes\cO_X
\longrightarrow\cR^\vee\longrightarrow0.
\]
Hence, the complete symmetric classes of $\cF^\vee$ satisfy
\[
s_0(\cF^\vee)=1,\qquad s_1(\cF^\vee)=-r,
\qquad s_j(\cF^\vee)=0\quad(j\geq2).
\]
By \eqref{eq:proj-gysin},
\[
p_*(\xi^{13})=1,\qquad p_*(\xi^{14})=-r,
\qquad p_*(\xi^{13+j})=0\quad(j\geq2),
\]
where $p:Z\to X$.

For the coefficient of $\sigma_6$ we pair with $\sigma_{13,7}$.  Since
\[
s_{13,7}(r,\xi)=\sum_{j=0}^{6}r^{13-j}\xi^{7+j},
\]
only the terms with $\xi$-exponent at least thirteen survive.  A direct application of the preceding pushforwards gives
\[
\begin{aligned}
p_*\bigl((2\xi+r)s_{13,7}(r,\xi)\bigr)
&=2r^8+2r^7(-r)+r^8\\
&=r^8.
\end{aligned}
\]
The class $r$ is the Pl\"ucker hyperplane class on $X=\Gr(2,6)$, and
\[
\int_Xr^8=\deg\Gr(2,6)=14.
\]
Thus the coefficient of $\sigma_6$ is fourteen.

For the coefficient of $\sigma_{5,1}$ we pair with $\sigma_{12,8}$.  Here
\[
s_{12,8}(r,\xi)=\sum_{j=0}^{4}r^{12-j}\xi^{8+j},
\]
and only the final term, after multiplication by $2\xi$, reaches exponent thirteen.  Hence
\[
p_*\bigl((2\xi+r)s_{12,8}(r,\xi)\bigr)=2r^8,
\]
whose integral is twenty-eight.  For the remaining dual classes $\sigma_{11,9}$ and $\sigma_{10,10}$, the largest possible $\xi$-exponent after multiplication by $2\xi+r$ is at most twelve, so their pairings vanish.  This proves
\[
f_*[2\xi+r]=14\sigma_6+28\sigma_{5,1}.
\]
The generic pencil in this irreducible divisor has a triple Pfaffian root at a unique rank-two member, and therefore has type $H_2(0)\oplus H_1(0)$.
\end{proof}

\section{Common-kernel loci}

For $1\leq k\leq3$, let
\[
C_k=\{L\in G:\dim\ker L\geq k\}.
\]
Let $B_k=\Gr(k,V)$, with tautological bundle $\cK_k$ and quotient bundle $\cE_k=V/\cK_k$ of rank $6-k$.  Define
\[
Y_k=\Gr_{B_k}\bigl(2,\Lambda^2\cE_k^\vee\bigr).
\]
The universal two-plane of alternating forms on $\cE_k$ pulls back to a pencil on $V$ vanishing on $\cK_k$, and yields a proper map
\[
f_k:Y_k\longrightarrow G.
\]

\begin{lemma}\label{lem:common-birational}
The map $f_k$ is birational onto $C_k$.  In particular,
\[
\codim_G C_1=5,\qquad \codim_G C_2=10,\qquad \codim_G C_3=15.
\]
\end{lemma}

\begin{proof}
A general pencil in $C_k$ has common kernel exactly $k$-dimensional, so the kernel is recovered uniquely from the pencil.  The dimensions are
\[
\dim Y_k=k(6-k)+2\left(\binom{6-k}{2}-2\right),
\]
which equal $21,16,11$ for $k=1,2,3$, while $\dim G=26$.
\end{proof}

The generic types in $C_1,C_2,C_3$ are respectively
\[
M_2\oplus M_0,\qquad
M_0^2\oplus H_1(0)\oplus K_1,
\qquad
M_1\oplus M_0^3.
\]

\begin{proposition}\label{prop:common-kernel}
The common-kernel orbit closures have classes
\[
[C_1]=12\sigma_{4,1}+30\sigma_{3,2},
\]
\[
[C_2]=3\sigma_{9,1}+15\sigma_{8,2}+40\sigma_{7,3}+70\sigma_{6,4}+91\sigma_{5,5},
\]
\[
[C_3]=56\sigma_{9,6}+70\sigma_{8,7}.
\]
\end{proposition}

\begin{proof}
Let
\[
p_k:Y_k=\Gr_{B_k}(2,\cF_k)\longrightarrow B_k,
\qquad
\cF_k=\Lambda^2\cE_k^\vee,
\]
and let $\cT_k$ be the tautological pencil.  Put
\[
N_k=\rk\cF_k=\binom{6-k}{2}.
\]
Because $f_k^*\cS=\cT_k$, the coefficient of $\sigma_{a,b}$ in $[C_k]$ is
\[
\int_{Y_k}s_{13-b,13-a}(\cT_k^\vee).
\]
Applying \cref{lem:schur-gysin} gives the uniform formula
\begin{equation}\label{eq:common-coeff}
[\sigma_{a,b}]\,[C_k]
=
\int_{B_k}
 s_{15-N_k-b,\,15-N_k-a}(\Lambda^2\cE_k),
\end{equation}
with the convention that a negative part gives zero.

We now evaluate the right-hand side by \eqref{eq:grass-coeff}.  For the sake of clarity, the complete calculation is displayed below.  In each row, $\mu$ is the partition appearing on the right of \eqref{eq:common-coeff}; the entry is
\[
[x_1^5x_2^4\cdots x_{6-k}^{k}]
\left(
 s_\mu(\{x_i+x_j\}_{i<j})
 \prod_{i<j}(x_i-x_j)
\right).
\]
Using \eqref{eq:complete-generating} and \eqref{eq:JT-two-row} yields
\[
\begin{array}{c|c|rrrrrr}
 k&N_k&\multicolumn{6}{c}{\mu\ \text{: its integral}}\\
\hline
1&10&(5):0&(4,1):12&(3,2):30&&&\\
2&6 &(8):3&(7,1):15&(6,2):40&(5,3):70&(4,4):91&\\
3&3 &(9):0&(8,1):0&(7,2):0&(6,3):56&(5,4):70.
\end{array}
\]
For example, when $k=1$ the base is $\Pj^5$, and the same expansion reads
\[
s_5(\Lambda^2\cE_1)=0,
\quad
s_{4,1}(\Lambda^2\cE_1)=12h^5,
\quad
s_{3,2}(\Lambda^2\cE_1)=30h^5,
\]
where $h=c_1(\cO_{\Pj^5}(1))$.

For $k=1$, the candidate codimension-five basis elements are
$\sigma_5,\sigma_{4,1},\sigma_{3,2}$, and \eqref{eq:common-coeff} associates to them $(5),(4,1),(3,2)$.  This gives
\[
[C_1]=12\sigma_{4,1}+30\sigma_{3,2}.
\]
For $k=2$, the class $\sigma_{10}$ corresponds to a negative second part and therefore has coefficient zero; the remaining five entries are the second row of the table.  For $k=3$, the possible classes are
\[
\sigma_{13,2},\sigma_{12,3},\sigma_{11,4},
\sigma_{10,5},\sigma_{9,6},\sigma_{8,7}.
\]
The first is excluded by a negative part in \eqref{eq:common-coeff}, and the next three correspond to $(9),(8,1),(7,2)$ and vanish.  The final two give $56$ and $70$.  This proves all three formulas.
\end{proof}

\subsection{A rank-two member after quotienting by the common kernel}

Let $B=\Pj(V)$, with tautological line $\cK$ and quotient $\cE=V/\cK$ of rank five.  Over $B$ form
\[
g:H=\Gr_B(2,\cE^\vee)\longrightarrow B,
\]
with tautological rank-two bundle $\cA$.  The line
\[
\cR=\Lambda^2\cA\subset\Lambda^2\cE^\vee
\]
is a decomposable rank-two alternating form on the quotient.  Put
\[
\cF=(\Lambda^2\cE^\vee)/\cR,
\qquad
p:Z=\Pj_H(\cF)\longrightarrow H.
\]
A point of $Z$ consists of a common kernel line, a rank-two member of the quotient pencil, and a second direction modulo that member.  Hence there is a proper map
\[
f:Z\longrightarrow G.
\]
The source is smooth of dimension
\[
5+\dim\Gr(2,5)+8=5+6+8=19.
\]
A general pencil in the image has common kernel exactly one-dimensional and a unique rank-two member after quotienting; these data recover the point of $Z$.  Thus $f$ is birational onto the closure of the type $M_1\oplus M_0\oplus H_1$.

\begin{proposition}\label{prop:common-ranktwo}
The closure of the type $M_1\oplus M_0\oplus H_1$ has class
\[
42\sigma_{6,1}+84\sigma_{5,2}.
\]
\end{proposition}

\begin{proof}
Set
\[
r=c_1(\cR^\vee),\qquad \xi=c_1(\cD^\vee),
\]
where $\cD\subset p^*\cF$ is the tautological direction line.  Then $f^*\cS^\vee$ has Chern roots $r,\xi$, so $f^*\sigma_{u,v}=s_{u,v}(r,\xi)$.

The bundle $\cF$ has rank nine.  From
\[
0\longrightarrow\cF^\vee\longrightarrow\Lambda^2\cE
\longrightarrow\cR^\vee\longrightarrow0
\]
we obtain, for $m\geq0$,
\begin{equation}\label{eq:F-complete}
s_m(\cF^\vee)
=s_m(\Lambda^2\cE)-r\,s_{m-1}(\Lambda^2\cE),
\end{equation}
where $s_{-1}=0$.  Expanding the two-row Schur polynomial and using \eqref{eq:proj-gysin} gives
\begin{equation}\label{eq:ranktwo-projective-push}
p_*s_{u,v}(r,\xi)
=
\sum_{\substack{0\leq j\leq u-v\\v+j\geq8}}
 r^{u-j}
 \left(
 s_{v+j-8}(\Lambda^2\cE)
-rs_{v+j-9}(\Lambda^2\cE)
 \right).
\end{equation}

It remains to push powers of $r$ through $g$.  If $z_1,z_2$ are the Chern roots of $\cA^\vee$, then $r=z_1+z_2$, and the elementary Schur expansion
\begin{equation}\label{eq:power-schur}
(z_1+z_2)^m
=
\sum_{b=0}^{\lfloor m/2\rfloor}
\left(\binom mb-\binom m{b-1}\right)
 s_{m-b,b}(z_1,z_2)
\end{equation}
(with $\binom m{-1}=0$), followed by \cref{lem:schur-gysin} for the rank-five bundle $\cE^\vee$, yields
\begin{equation}\label{eq:r-power-push}
g_*(r^m)
=
\sum_{b=3}^{\lfloor m/2\rfloor}
\left(\binom mb-\binom m{b-1}\right)
 s_{m-b-3,b-3}(\cE).
\end{equation}

Substituting \eqref{eq:F-complete}--\eqref{eq:r-power-push} into \eqref{eq:ranktwo-projective-push}, and finally using \eqref{eq:grass-coeff} on $B=\Pj^5$, gives
\[
\begin{array}{c|c|c}
(a,b)&(u,v)=(13-b,13-a)&g_*p_*s_{u,v}(r,\xi)\\
\hline
(7,0)&(13,6)&0\\
(6,1)&(12,7)&42h^5\\
(5,2)&(11,8)&84h^5\\
(4,3)&(10,9)&0,
\end{array}
\]
where $h=c_1(\cO_{\Pj^5}(1))$.  Each row is a finite application of the displayed formulas; for instance, all complete symmetric classes needed in \eqref{eq:ranktwo-projective-push} have degree at most five and are obtained from \eqref{eq:complete-generating}.  Since $\int_{\Pj^5}h^5=1$, the only nonzero Poincar\'e pairings are $42$ and $84$, dual to $\sigma_{6,1}$ and $\sigma_{5,2}$.
\end{proof}

\subsection{A repeated quotient Pfaffian}

On $Y_2$, the quotient space $\cE_2$ has rank four, so the Pfaffian of the universal quotient pencil is a binary quadratic with values in $\det\cE_2^\vee$.  Its discriminant is a section of
\[
(\det\cT^\vee)^2\otimes(\det\cE_2^\vee)^2,
\]
where $\cT$ is the tautological rank-two bundle on $Y_2$.  Its zero divisor maps birationally to the closure of $M_0^2\oplus H_2$.

\begin{proposition}\label{prop:common-double}
The closure of the type $M_0^2\oplus H_2$ has class
\[
6\sigma_{10,1}+24\sigma_{9,2}+50\sigma_{8,3}+60\sigma_{7,4}+42\sigma_{6,5}.
\]
\end{proposition}

\begin{proof}
Write
\[
u=c_1(\cT^\vee),\qquad e=c_1(\det\cE_2^\vee),
\]
so that the divisor class is $2u+2e$.  Let
\[
p:Y_2=\Gr_{B_2}(2,\Lambda^2\cE_2^\vee)\longrightarrow B_2=\Gr(2,V).
\]
The bundle $\Lambda^2\cE_2^\vee$ has rank six.  Put
\[
A_\mu=s_\mu(\Lambda^2\cE_2).
\]
By \cref{lem:schur-gysin},
\[
p_*s_{r,s}(\cT^\vee)=A_{r-4,s-4}
\]
when $s\geq4$, and the pushforward vanishes otherwise.  Pieri's rule gives
\[
u\,s_{r,s}(\cT^\vee)
=s_{r+1,s}(\cT^\vee)+s_{r,s+1}(\cT^\vee),
\]
with the second term omitted when $r=s$.

The Chern numbers needed on $B_2$ are, by \eqref{eq:grass-coeff},
\[
\begin{array}{c|rrrrr}
\mu&(8)&(7,1)&(6,2)&(5,3)&(4,4)\\
\hline
\int A_\mu&3&15&40&70&91,
\end{array}
\]
and
\[
\begin{array}{c|rrrr}
\mu&(7)&(6,1)&(5,2)&(4,3)\\
\hline
\int eA_\mu&-6&-30&-80&-140.
\end{array}
\]
For example, these numbers are the coefficients prescribed by \eqref{eq:grass-coeff} after evaluating the Schur polynomials in the six roots $x_i+x_j$ of $\Lambda^2\cE_2$; no other characteristic numbers enter.

We now pair with the dual basis.  The calculation, including the separate contributions of $2u$ and $2e$, is
\[
\begin{array}{c|c|r|r|r}
(a,b)&(r,s)=(13-b,13-a)&\int 2u\,s_{r,s}&\int2e\,s_{r,s}&\text{total}\\
\hline
(11,0)&(13,2)&0&0&0\\
(10,1)&(12,3)&6&0&6\\
(9,2)&(11,4)&36&-12&24\\
(8,3)&(10,5)&110&-60&50\\
(7,4)&(9,6)&220&-160&60\\
(6,5)&(8,7)&322&-280&42.
\end{array}
\]
To illustrate one row, for $(r,s)=(10,5)$ we have
\[
p_*(u s_{10,5})=A_{7,1}+A_{6,2},
\qquad
p_*s_{10,5}=A_{6,1},
\]
so the pairing is
\[
2(15+40)+2(-30)=50.
\]
The other rows follow identically.  This proves the asserted class.
\end{proof}

\bigskip

\noindent\textsc{Department of Mathematics, Harvard University, 1 Oxford Street, Cambridge, MA 02138, USA}\par
\noindent\textit{Email address:} \texttt{akrishna@college.harvard.edu}
\end{document}